\documentclass[12pt,reqno]{amsart}
\usepackage[margin=2.2 cm]{geometry}
\usepackage{amsfonts,amssymb,amsthm, mathtools}
\usepackage{amsmath}
\usepackage{datetime}
\usepackage{tabularx}
\everymath{\displaystyle}
\usepackage{xcolor}
\usepackage{dsfont}
\usepackage{enumerate,physics}
\usepackage{multirow}
\usepackage{float}
\everymath{\displaystyle}

\definecolor{fgreen}{RGB}{44,144, 14}

\numberwithin{equation}{section} 
\newtheorem{theorem}{Theorem}[section]

\newtheorem{lemma}[theorem]{Lemma} 
\theoremstyle{definition}
\newtheorem{definition}[theorem]{Definition}

\usepackage[colorlinks=true,citecolor=cyan,  urlcolor=blue,linkcolor=blue]{hyperref}

\def\C{\mathbb C}

\def\H{\mathbb H}

\def\N{\mathbb N}

\def\P{\mathbb P}

\def\C{\mathbb {C}}
\def\N{\mathbb {N}}
\def\H{\mathbb {H}}

\begin{document} 
\title[Equicontinuity Regions]{Classification of Quaternionic Projective Transformations by Equicontinuity Regions}
	\author[S.  Dutta, K. Gongopadhyay and R. Mondal]{Sandipan Dutta, Krishnendu Gongopadhyay and 
		Rahul Mondal}
	
	\address{Mizoram University,
		Tanhril, Aizawl 796004, Mizoram, India}
	\email{sandipandutta98@gmail.com, mzut330@mzu.edu.in}
	
	\address{Indian Institute of Science Education and Research (IISER) Mohali,
		Knowledge City,  Sector 81, S.A.S. Nagar 140306, Punjab, India}
	\email{krishnendug@gmail.com, krishnendu@iisermohali.ac.in}

	\address{Indian Institute of Science Education and Research (IISER) Mohali,
		Knowledge City,  Sector 81, S.A.S. Nagar 140306, Punjab, India}
	\email{canvas.rahul@gmail.com}
	
	\date{ @\currenttime , \today}
\begin{abstract}
We describe the equicontinuity regions of cyclic subgroups of the quaternionic projective linear group $\mathrm{PSL}(n+1,\mathbb{H})$. We show that these regions depend solely on the dynamical type of the generator $g$, i.e. whether $g$ is elliptic, parabolic, loxodromic or loxoparabolic. This yields an analytic interpretation of the dynamical classification of the  elements. In particular, elliptic cyclic groups act equicontinuously on all of the quaternionic projective space, while for the parabolic, loxodromic, and loxoparabolic elements, the equicontinuity region is determined by explicit quaternionic projective subspaces arising from the generator's Jordan form.
\end{abstract}
    	\maketitle 

	\section{Introduction} 

The dynamics of projective actions of ${\rm PSL}(2, \C)$  has long constituted a central theme in the study of Kleinian groups and geometric group theory. For discrete groups acting on higher dimensional complex projective space, the \emph{equicontinuity region} plays a fundamental role as an invariant that generalizes the classical domain of discontinuity (see, for instance, Cano--Seade~\cite{can2}, Cano--Navarrete--Seade~\cite{cns}, and Cano--Loeza--Ucan Puc~\cite{cano_main}). It represents the largest open subset on which the group action induces a regular, well-behaved dynamics. It has proven to be a powerful data in understanding limit sets, orbit structures, and global geometric properties of complex projective transformations, cf. \cite{cns}. 

Recently, the {quaternionic analogues of complex Kleinian groups} has gained some attention which is motivated by the rich interplay between quaternionic hyperbolic geometry and the dynamics of projective transformations over the quaternions. The quaternionic setting has a new layer of complexity due to the non-commutativity of the quaternionic multiplication. This alters the framework that is used in the complex case.  Many of the tools and results in complex projective dynamics cannot be directly transferred, and must instead be reformulated to accommodate the subtleties of quaternionic analysis.

In our earlier works~\cite{dgm} and \cite{dgl},  we initiated a systematic study of the dynamics of \emph{quaternionic cyclic groups} and obtained a detailed description of their Kulkarni limit sets. This investigation extended the corresponding results established in the complex setting by Cano, Loeza, and Ucan-Puc~\cite{cano_main}. It is worth recalling that the study of complex Kleinian groups was originally initiated by Seade and Verjovsky~\cite{sv99, sv01}, and subsequently developed in a series of works by Cano, Barrera, Navarrete, and others, see the comprehensive monograph~\cite{cns} or the survey \cite{ bcnss} for an overview. The theory that has emerged in the complex case offers deep insights into the global dynamics of projective transformations, but its quaternionic counterpart remains relatively unexplored.

The notion of \emph{equicontinuity}, which we aim to investigate in this paper, is rooted in the classical Arzel\'{a}--Ascoli theorem, providing a compactness criterion for families of functions. In~\cite{can}, Cano and Seade 
introduced the concept of the equicontinuity region for ${\rm PSL}(3, \C)$ action on the complex projective space  $\mathbb{P}^2_\mathbb{C}$ and established its relation with the Kulkarni limit set. This theory was subsequently generalized to higher-dimensional complex projective spaces $\mathbb{P}^n_\mathbb{C}$ in~\cite{cano_main}, where the authors provided a classification of the equicontinuity regions for cyclic subgroups of $\mathrm{PSL}(n+1, \mathbb{C})$. This gave a  dynamical classification of the complex projective transformations. 

The principal aim of the present article is to {extend this work to the quaternionic setting} in order to provide a complete classification of equicontinuity regions for cyclic subgroups of $\mathrm{PSL}(n+1, \mathbb{H})$. To formulate our main results, we need the following notions. 

\begin{definition}
A family $\mathcal{F}$ of maps from a topological space into $\mathbb{P}^n_\mathbb{H}$ is said to be \emph{normal} if every sequence in $\mathcal{F}$ contains a subsequence that converges uniformly on compact subsets.
\end{definition}

Using this concept, we define the equicontinuity region for projective actions over the quaternions.

\begin{definition}
\label{eqcnt}
Let $G$ be a subgroup of $\mathrm{PSL}(n+1, \mathbb{H})$. The \emph{region of equicontinuity} of $G$, denoted $\mathrm{Eq}(G)$, is the set of points $z \in \mathbb{P}^n_\mathbb{H}$ such that there exists an open neighborhood $U$ of $z$ for which the restricted family
\[
G_U = \{ g|_U : g \in G \}
\]
is normal. 

\medskip Intuitively, $\mathrm{Eq}(G)$ consists of those points in $\mathbb{P}^n_\mathbb{H}$ where the group action behaves in a uniformly controlled manner, where orbits do not exhibit erratic divergence or collapse, and hence where local compactness and convergence properties are preserved.
\end{definition}

In order to classify the equicontinuity regions of cyclic quaternionic projective groups, we analyze elements $g \in \mathrm{PSL}(n+1, \mathbb{H})$ via their quaternionic Jordan canonical forms, and determine the corresponding equicontinuity region $\mathrm{Eq}(G)$ for the cyclic subgroup $G = \langle g \rangle$. The dynamical nature of $g$ can be described by the following classification.
\begin{definition}
\label{def:classification}
\label{trichotomy}
Let $g \in \mathrm{PSL}(n+1, \mathbb{H})$. Then:
\begin{enumerate}[(i)]
    \item $g$ is called \emph{elliptic} if it is semisimple and all its eigenvalue classes are represented by quaternions of unit modulus;
    \item $g$ is called \emph{parabolic} if it is not semisimple and all its eigenvalues have unit modulus;
     \item $g$ is called \emph{loxodromic} if it is semisimple but not all eigenvalue classes are represented by unit modulus quaternions;
    \item $g$ is called \emph{loxoparabolic} if it is not semisimple and at least one eigenvalue has non-unit modulus.
\end{enumerate}
\end{definition}

Our approach builds upon the {quaternionic Jordan canonical form} together with the method of {pseudo-projective limits}. Our methods are inspired by the strategy developed in the complex setting by Cano and Seade in~\cite{can}. The quaternionic structure introduces new phenomena not present in the complex case, arising from the non-commutative nature of the underlying field and the resulting behavior of similarity classes of eigenvalues.
\begin{theorem}
Let $\gamma \in \mathrm{PSL}(n+1,\mathbb{H})$ be a projective transformation and let
\[
\Gamma := \langle \gamma \rangle \subset \mathrm{PSL}(n+1,\mathbb{H})
\]
be the cyclic subgroup generated by $\gamma$. Denote by $\mathrm{Eq}(\Gamma)\subset\mathbb{P}^n_{\mathbb{H}}$ the equicontinuity set of the family $\Gamma$. Then the following holds. 
\begin{enumerate}[(i)]
    \item The element $\gamma$ is \emph{elliptic} if and only if $\mathrm{Eq}(\Gamma)=\mathbb{P}^n_{\mathbb{H}}$.
    
    \item The element $\gamma$ is \emph{parabolic} if and only if the complement $\mathbb{P}^n_{\mathbb{H}}\setminus\mathrm{Eq}(\Gamma)$ is a single projective subspace $L$ (see Theorem ~\ref{th:par-main} for a precise description).
    
    \item The element $\gamma$ is \emph{loxodromic} if and only if it is diagonalizable and the complement $\mathbb{P}^n_{\mathbb{H}}\setminus\mathrm{Eq}(\Gamma)$ can be described as the union of two distinct proper projective subspaces $L_1, L_2$ of $\mathbb{P}^n_{\mathbb{H}}$ (see Theorem ~\ref{th:loxodromic} for a precise description).
    
    \item The element $\gamma$ is \emph{loxoparabolic} if and only if it is not diagonalizable and the complement $\mathbb{P}^n_{\mathbb{H}}\setminus\mathrm{Eq}(\Gamma)$ can be described as the union of two distinct proper projective subspaces $L_1, L_2$ of $\mathbb{P}^n_{\mathbb{H}}$ (see Theorem ~\ref{th:loxoparabolic} for a precise description).
\end{enumerate}
\end{theorem}

Our results indicate that the description of the \emph{dynamical type} of the generator $g$ is classified by the equicontinuity region of $g$, and this provides an analytic classification of the elliptic, parabolic, loxodromic and loxoparabolic elements in ${\rm PSL}(n, \H)$.  It may be noted that there have been attempts to relate the classification of transformations to their algebraic data, for example, see \cite{dgl2}, \cite{kl}, \cite{ps}, \cite{go}. Specifically, we establish that elliptic cyclic groups act equicontinuously on the entire quaternionic projective space, whereas in the parabolic and loxodromic cases, the equicontinuity region is described by explicit quaternionic projective subspaces determined by the structure of the Jordan blocks associated with the generator. We hope that this work (along with \cite{dgm}) lays the foundation for further investigations aimed at understanding the global dynamics of quaternionic projective groups and their invariants.

		\subsection{Structure of the paper:} Following an introduction in the first section, we present the necessary preliminaries in Section \ref{sec:prelim}. Section \ref{sec:elliptic} details the equicontinuity region for elliptic elements of 
    $\mathrm{PSL}(n,\H)$, while Section \ref{sec:parabolic} analyzes parabolic elements. Finally, we compute the equicontinuity regions for loxodromic and loxoparabolic elements in Sections \ref{sec:loxodromic} and \ref{sec:loxo}, respectively.
    	\subsection{Notations}\label{notations}
    	\begin{enumerate}[(i)]
		\item  $\mathrm{D}(\lambda_1,\lambda_2,\ldots,\lambda_n)$ denotes a $n\times n$ diagonal matrix whose diagonal entries are $\lambda_1,\lambda_2,$ $\ldots,\lambda_n$.
		If $\lambda_1=\lambda_2=\ldots=\lambda_n$ then we simply denote as $\mathrm{D}(\lambda,n)$.
		\item $ \mathrm{J}(\lambda,n+1)$ denotes a $(n+1)\times (n+1)$ matrix which have $\lambda$ at the diagonal and 1 at the super diagonal.
		\par Note that, 
		\begin{equation*}
			\mathrm{J}(\lambda,n+1)^m=\begin{bsmallmatrix}
				\lambda^m & {m \choose 1}\lambda^{m-1} & {m \choose 2}\lambda^{m-2} & \ldots & {m \choose n}\lambda^{m-n}\\&  \lambda^m & {m \choose 1}\lambda^{m-1} & \ldots & {m \choose n-1}\lambda^{m-n+1}\\ &   & \ddots & \ddots & \vdots\\
				&   & &\ddots & {m \choose 1}\lambda^{m-1}  \\ & &   &  & \lambda^m 
			\end{bsmallmatrix}=\left[ {m\choose j-i}\lambda^{m-j+1}\right]_{1\leq i,j\leq n+1}.
		\end{equation*}
        Also
        \begin{equation*}
\mathrm{J}(\lambda,n+1)^{-m}
=
\begin{bsmallmatrix}
\lambda^{-m} & -\binom{m}{1}\lambda^{-m-1} & \binom{m+1}{2}\lambda^{-m-2} & \cdots & (-1)^n \binom{m+n-1}{n} \lambda^{-m-n} \\
0 & \lambda^{-m} & -\binom{m}{1}\lambda^{-m-1} & \cdots & (-1)^{n-1} \binom{m+n-2}{n-1} \lambda^{-m-(n-1)} \\
0 & 0 & \lambda^{-m} & \ddots & \vdots \\
\vdots & \vdots & \vdots & \ddots & -\binom{m}{1}\lambda^{-m-1} \\
0 & 0 & 0 & 0 & \lambda^{-m}
\end{bsmallmatrix}.
        \end{equation*}
       \textbf{Note.}
For the Jordan block \(\mathrm{J}(\lambda,n+1)\), after normalizing by the maximal
entry in \(J(\lambda,n+1)^{\pm m}\), one has
\[
\frac{\bigl(\mathrm{J}(\lambda,n+1)^{\pm m}\bigr)_{i,j}}
     {\bigl(\mathrm{J}(\lambda,n+1)^{\pm m}\bigr)_{1,n+1}}
\;\longrightarrow\; 0
\qquad \text{as } m\to\infty,
\]
for every \((i,j)\neq (1,n+1)\).\\
In words: for both positive and negative powers of the Jordan block,
every entry becomes negligible compared to the \((1,n+1)\)-entry
as \(m\to\infty\).\\

		\item $\begin{bmatrix}
			\mathrm{D}(\lambda,l) & \\ & \mathrm{J}(\mu,k)
		\end{bmatrix}$ denotes a $(l+k)\times (l+k)$ dimensional block diagonal matrix which have blocks $\mathrm{D}(\lambda,l)$ and $\mathrm{J}(\mu,k)$.
		\item   $\mathrm{E}_{m,n}$ be a $n\times n$ matrix such that all its entries except the $(m,n)^{\text{th}}$ entry is 0 and the $(m,n)^{\text{th}}$ entry is 1.
		\item $L{\{p_1,p_2,\ldots,p_n\}}$ or, $\langle  p_1,p_2,\ldots,p_n \rangle $ denotes a subspce of the projective space $\P^{n}_\H$ generated by $p_1,p_2,\ldots,p_n\in \P^{n}_\H$. Similarly, $\langle U\rangle$ denotes the subspace generated by the set $U$.  
	\end{enumerate}
   \section{Preliminaries} \label{sec:prelim}
  Let $\mathbb{H}$ denote the division ring of Hamiltonâ€™s quaternions. Every quaternion $a \in \mathbb{H}$ can be expressed as
\[
a = a_0 + a_1 i + a_2 j + a_3 k, \qquad a_0,a_1,a_2,a_3 \in \mathbb{R},
\]
where the units satisfy $i^2 = j^2 = k^2 = ijk = -1$. The conjugate of $a$ is
\[
\bar{a} = a_0 - a_1 i - a_2 j - a_3 k.
\]
The real subspace $\mathbb{R}\oplus \mathbb{R}i$ is identified with the complex plane $\mathbb{C}$. For further background on quaternionic linear algebra and matrix theory, see \cite{rodman,FZ}.

\subsection*{Eigenvalues of matrices in  $\mathrm{M}(n,\mathbb{H})$}

\begin{definition}\label{def-eigen-M(n,H)-rew}
Let $A \in \mathrm{M}(n,\mathbb{H})$. A non-zero vector $v \in \mathbb{H}^n$ is a (right) eigenvector of $A$ with (right) eigenvalue $\lambda \in \mathbb{H}$ if
\[
A v = v \lambda.
\]
\end{definition}

Right eigenvalues occur in similarity classes: if $A v = v \lambda$, then for any $\mu \in \mathbb{H}^{\times}$, the vector $v\mu$ is also an eigenvector with eigenvalue $\mu^{-1}\lambda \mu$. Each similarity class contains a unique complex number with non-negative imaginary part, which we take as its representative and refer to simply as an eigenvalue.

\subsection*{Complex Embedding and Quaternionic Determinant}

Any $A \in \mathrm{M}(n,\mathbb{H})$ can be written as
\[
A = A_1 + A_2 j, \qquad A_1, A_2 \in \mathrm{M}(n,\mathbb{C}).
\]
The standard embedding
\[
\Phi : \mathrm{M}(n,\mathbb{H}) \longrightarrow \mathrm{M}(2n,\mathbb{C})
\]
is given by
\begin{equation}\label{eq-embed-phi-rew}
\Phi(A) = 
\begin{bmatrix}
A_1 & A_2 \\
- \overline{A_2} & \overline{A_1}
\end{bmatrix}.
\end{equation}

\begin{definition}
For $A \in \mathrm{M}(n,\mathbb{H})$, the quaternionic determinant of $A$ is
\[
{\det}_{\mathbb{H}}(A) := \det(\Phi(A)).
\]
This definition is independent of the embedding $\Phi$ by the Skolem--Noether theorem.
\end{definition}

The corresponding Lie groups are
\[
\mathrm{GL}(n,\mathbb{H}) = \{g \mid {\det}_{\mathbb{H}}(g)\neq 0\}, \qquad
\mathrm{SL}(n,\mathbb{H}) = \{g \mid {\det}_{\mathbb{H}}(g)=1\}.
\]

\subsection*{Jordan Theory over Quaternions}

\begin{definition}{\cite{rodman}}
A Jordan block $\mathrm{J}(\lambda,m)$ is the $m \times m$ matrix having $\lambda$ on its diagonal, $1$ on its superdiagonal, and zeros elsewhere. A direct sum of Jordan blocks is called a Jordan form.
\end{definition}

\begin{lemma}[Quaternionic Jordan form, cf.~{\cite[Th.~5.5.3]{rodman}}]\label{lem-Jordan-M(n,H)-rew}
For every $A \in \mathrm{M}(n,\mathbb{H})$, there exists $S \in \mathrm{GL}(n,\mathbb{H})$ such that
\[
S A S^{-1} = \mathrm{J}(\lambda_1,m_1) \oplus \cdots \oplus \mathrm{J}(\lambda_k,m_k),
\]
where each $\lambda_i$ is a complex number with non-negative imaginary part. This decomposition is unique up to permutation of blocks.
\end{lemma}

\subsection*{Quaternionic Projective Space $\mathbb{P}^n_{\mathbb{H}}$}

Consider $\mathbb{H}^{n+1}$ as a right $\mathbb{H}$-vector space. The quaternionic projective space is defined by
\[
\mathbb{P}^n_{\mathbb{H}} = \big( \mathbb{H}^{n+1}\setminus\{0\} \big) / \sim,
\qquad
z \sim w \iff z = w \alpha \ \text{for some }\alpha \in \mathbb{H}^\times.
\]
Let $\mathbb{P}$ denote the quotient map. A subset $W\subseteq \mathbb{P}^n_{\mathbb{H}}$ is a $k$-dimensional projective subspace if there exists a $(k+1)$-dimensional $\mathbb{H}$-linear subspace $\widetilde{W}\subseteq \mathbb{H}^{n+1}$ such that
\[
\mathbb{P}(\widetilde{W}\setminus\{0\}) = W.
\]
Projective $1$-subspaces are called \emph{lines}. For
\[
p=(x_1,\dots,x_{n+1}) \in \mathbb{H}^{n+1},
\quad
[x_1:\cdots:x_{n+1}] := \mathbb{P}(p),
\]
and for $\alpha \neq 0$,
\[
[x_1:\cdots:x_{n+1}] = [x_1\alpha : \cdots : x_{n+1}\alpha].
\]

Given $S \subseteq \mathbb{P}^n_{\mathbb{H}}$,
\[
\langle S\rangle = \bigcap \{ H : H \ \text{is a projective subspace containing } S\}.
\]
If $p_1,\dots,p_m$ are distinct, then $\langle p_1,\dots,p_m\rangle$ is the unique projective subspace containing them; for $m=2$ it is the quaternionic projective line $L\{p,q\}$ (cf.~\cite{dgl}).

\subsection*{Projective Transformations}

For $\gamma \in \mathrm{GL}(n+1,\mathbb{H})$,
\[
\gamma(\mathbb{P}(z)) = \mathbb{P}(\gamma z), \qquad z\neq 0.
\]
Scalar multiples act trivially, so
\[
\mathrm{PSL}(n+1,\mathbb{H})
= \mathrm{SL}(n+1,\mathbb{H}) / \mathcal{Z}(\mathrm{SL}(n+1,\mathbb{H})),
\qquad
\mathcal{Z} = \{\pm I_{n+1}\}.
\]
Each element has exactly two lifts $\widetilde{\gamma}$ and $-\widetilde{\gamma}$. Using quaternionic Jordan theory, one obtains a well-defined classification of elements as elliptic, loxodromic, or parabolic (see Definition~\ref{def:classification}).

\subsection*{Pseudo-Projective Transformations}

Following Cano--Seade \cite{can2}, we extend the notion of pseudo-projective transformations to the quaternionic setting. Let $\widetilde{M}:\mathbb{H}^{n+1}\to\mathbb{H}^{n+1}$ be a non-zero $\mathbb{H}$-linear map. It induces
\[
M : \mathbb{P}^n_{\mathbb{H}} \setminus \ker(M) \longrightarrow \mathbb{P}^n_{\mathbb{H}},
\qquad
M(\mathbb{P}(v)) = \mathbb{P}(\widetilde{M}(v)),
\]
where $\ker(M)$ is the projectivization of $\ker(\widetilde{M})\setminus\{0\}$. We call $M$ a \emph{pseudo-projective transformation} and set
\[
\mathrm{QP}(n+1,\mathbb{H})
= \{ M \mid \widetilde{M} \ \text{is a non-zero }\mathbb{H}\text{-linear transformation}\}.
\]
The group $\mathrm{PSL}(n+1,\mathbb{H})$ embeds naturally into $\mathrm{QP}(n+1,\mathbb{H})$.

Pseudo-projective transformations, together with the classification of lifts in $\mathrm{PSL}(n+1,\mathbb{H})$, will be used to study of equicontinuty region .

   \begin{definition}\label{def_limit}
       Let $G$ be a subgroup of $\mathrm{PSL}(n+1,\H)$. Then an element $\gamma\in \mathrm{QP}(n+1,\H)$ is called a limit of the group $G$ if there exists a sequence of distinct elements $(\gamma_m)_{m\in \N}\subset G$ such that $\gamma_m\xrightarrow{m\rightarrow\infty} \gamma$.
   \end{definition}
   If the Definition \ref{def_limit} holds, then the set of all the limits of $G$ is denoted by $Lim(G).$
   
   \subsection{Some useful results}
     \begin{lemma}\cite[Lemma~2.9]{dgm}
			\label{lem:p1}
			$(g_n)_{n\in \N}\subset \mathrm{PSL}(n+1,\H)$ be a sequence, then there exists a sub sequence $(g_{n_k})_{k\in \N}\subset \mathrm{PSL}(n+1,\H)$ and an element $g\in \mathrm{QP}(n+1,\H)$ such that $g_{n_k}\rightarrow g$ uniformly on the compact subsets of $\P^n_\H\setminus \ker(g)$.
		\end{lemma}
    \begin{lemma}\label{lem:ker0}
  Let $(\gamma_m)_{m\in \N}\subset \mathrm{PSL}(n+1,\H)$ be a sequence converging to $\gamma\in \mathrm{QP}(n+1,\H)\setminus \mathrm{PSL}(n+1,\H)$. Also assume $L\subset \P^n_\H$ be a projective subspace such that $L\cap \ker(\gamma)=\{x\}$ and $\dim(L)\geq \dim(\mathrm{Im}(\gamma))+1$. Then there is a subsequence $(\gamma_{m_k})_{k\in\N}\subset (\gamma_m)_{m\in\N}$ such that $\gamma_{m_k}(x_m)\xrightarrow{k\rightarrow\infty}y$, where $x_m\rightarrow x$ and $y\in M$, $M\subset \P^n_\H$ is a subset with $\dim(M)=\dim(L)$ and $M\cap \mathrm{Im}(\gamma)$ is non-empty.
   \end{lemma}
   \begin{proof}
       Let $\dim(L)=k$. As $Gr_\H(k,n)$ is compact hence there exists a subsequence $(\gamma_{n_k})$ and a subspace $M\in Gr_\H(k,n)$ such that $\gamma_{n_k}(L)\xrightarrow{k\rightarrow \infty}M$.
       \par Assume $y\in M\setminus \mathrm{Im}(\gamma)$. Then there exists a convergent sequence $(x_m)\subset L$ such that $\gamma_{n_k}(x_k)\xrightarrow{k\rightarrow \infty} y$. Now limit point of $(x_m)$ lies in $L$ as well as in $\ker(\gamma)$ hence it is $p$.
       \par From the facts that $\gamma_{n_k}\rightarrow \gamma$ and $L\cap \ker(\gamma)=\{p\}$ we have $M\cap \mathrm{Im}(\gamma)\neq \emptyset.$
   \end{proof}
   	
   \begin{lemma}
\label{lemma:ker}
Let $(\gamma_n) \subset \mathrm{PSL}(n+1, \mathbb{H})$ be a sequence that converges to a pseudo-projective transformation $\gamma \in \mathrm{QP}(n, \mathbb{H})$ uniformly on compact subsets of $\mathbb{P}^n_{\mathbb{H}} \setminus \ker(\gamma)$. Then the equicontinuity region of the sequence $\{\gamma_n\}$ is given by
\[
\mathrm{Eq}(\{\gamma_m : m \in \mathbb{N}\}) = \mathbb{P}^n_{\mathbb{H}} \setminus \ker(\gamma).
\]
\end{lemma}
\begin{proof} If the above statement is not true then there exists $p\in \ker(\gamma)\cap \mathrm{Eq}(\{\gamma_m : m \in \mathbb{N}\})$. Let $K\subset \mathbb{P}^n_{\mathbb{H}}$ be a projective subspace such that $K\cap \ker(\gamma)$ and $\dim (K)=\dim (\Im(\gamma))$.
	\par Take, $M=L\{K,x\}$ then by Lemma \ref{lem:ker0} there exists a subsequence $(\gamma_{m_k})\subset (\gamma_m)$ and $N\subset \mathbb{P}^n_{\mathbb{H}}$ such that $N\cap \Im(\gamma)\neq \emptyset$ with $\dim(M)=\dim(N)$. For every $y\in N$, there exists $(y_m)\subset  M$ such that $y_m\rightarrow p$ and $ \gamma_m(y_m)\rightarrow y$.\\
	 Let $z,w\in N\setminus \ker(\gamma)$ be two distinct point then we have $(z_m),(w_m)\subset N$ with $z_m,w_m\rightarrow p$ and $\gamma_m(z_m)\rightarrow z, \gamma_m(w_m)\rightarrow w$.
	\par From the definition of equicontinuity there is a subsequence $(\kappa_m)\subset (\gamma_{m_k})$ such that $\kappa_m\rightarrow \kappa$ in compact-open topology. Again as $p\in \mathrm{Eq}(\{\gamma_m : m \in \mathbb{N}\})$ hence $z=\kappa(p)=w$. It gives us a contradiction to the fact that $z$ and $w$ are distinct.
\end{proof}

   From the previous Lemmas and Definition \ref{def_limit} we can deduce the following theorem.
   \begin{theorem}\label{Th:ker}
      Let $G\subset \mathrm{PSL}(n+1,\H)$ be a subgroup and $g\in \mathrm{QP}(n+1,\H)$ be a limit of $G$ then 
      \begin{equation*}
          \mathrm{Eq}(G)=\P^n_\H \setminus \overline{\bigcup_{g\in Lim(G)}\ker(g)}.
      \end{equation*}
   \end{theorem}
   \section{Equicontinuity regions of cyclic subgroups generated by elliptic element} \label{sec:elliptic}
\begin{theorem}
\label{th:elliptic}
Let $\tilde{\gamma} \in \mathrm{PSL}(n+1, \mathbb{H})$ be an elliptic transformation with lift 
\[
\gamma = \mathrm{D}\big[e^{2\pi i \alpha_1}, \ldots, e^{2\pi i \alpha_{n+1}}\big] \in \mathrm{SL}(n+1, \mathbb{H}).
\]
Then the equicontinuity region of the cyclic group $\Gamma := \langle \tilde{\gamma} \rangle$ is given by
\[
\mathrm{Eq}(\Gamma) = \mathbb{P}^n_{\mathbb{H}}.
\]
\end{theorem}

\begin{proof}
We consider two cases depending on the rationality of the parameters $\alpha_i$.

\smallskip
\noindent
\textbf{Case 1.} Suppose that all $\alpha_i \in \mathbb{Q}$. Then there exists $n_0 \in \mathbb{N}$ such that 
\[
e^{2\pi i n_0 \alpha_i} = 1 \quad \text{for all } i = 1, 2, \ldots, n+1.
\]
Hence, $\gamma$ is of finite order, and so is $\Gamma$. That implies,
\[
\mathrm{Eq}(\Gamma) = \mathbb{P}^n_{\mathbb{H}}.
\]

\smallskip
\noindent
\textbf{Case 2.} Suppose that at least one of the $\alpha_i$, say $\alpha_k$, is irrational. Then there exists a subsequence $(\gamma^{n_k})$ converging to some 
\[
g \in \mathrm{Lim}(\Gamma).
\]
Since $|e^{2\pi i \alpha_i n_k}| = 1$ for all $i$, the limit map can be written as 
\[
g = \mathrm{D}[\alpha_1, \ldots, \alpha_{n+1}],
\]
where $|\alpha_i| = 1$ for every $i$. As $\ker(g) = \emptyset$, it follows from Lemma \ref{lemma:ker} that
\[
\mathrm{Eq}(\Gamma) = \mathbb{P}^n_{\mathbb{H}}.
\]
This completes the proof.
\end{proof}

\section{Equicontinuity regions of cyclic subgroups generated by parabolic element}\label{sec:parabolic}
In this section we shall compute the equicontinuity region for the cyclic subgroups generated by parabolic elements.

\begin{lemma}
\label{th:par1}
Let $\tilde{\gamma} \in \mathrm{PSL}(n+1, \mathbb{H})$ be a parabolic transformation whose lift is given by 
\[
\gamma = \mathrm{J}(\lambda, n+1) \in \mathrm{SL}(n+1, \mathbb{H}).
\]
Then the equicontinuity region of the cyclic group $\Gamma := \langle \tilde{\gamma} \rangle$ is 
\[
\mathrm{Eq}(\Gamma) = \P^n_\H\setminus L\{ e_1, e_2, \ldots, e_n \}.
\]
\end{lemma}

\begin{proof}
The proof is divided into two parts.

First, consider the powers of $\gamma$. By computing $\gamma^k$ and examining the sequence
\[
\frac{1}{{k \choose n}}\gamma^k,
\]
If a subsequence converges to a pseude projective limit then it must be of the form;
\[
g = a\,\mathrm{E}_{1,n+1},
\]
where $|a| = 1$.

Next, we analyze the inverses $\gamma^{-k}$. Consider the sequence
\[
\frac{1}{{k+n-1 \choose n}} \gamma^{-k}.
\]

Also , if a subsequence of $\frac{1}{{k+n-1 \choose n}} \gamma^{-k}$ converges to a pseudo projective limit then it must be of the type
\[
g = a\,\mathrm{E}_{1,n+1},
\]
where again $|a| = 1$. Therefore, if $g\in Lim(\langle \tilde{\gamma}\rangle)$ 
\\then \[
g = a\,\mathrm{E}_{1,n+1},
\] where $|a|=1$. Since 
\[
\ker(g) = L\{ e_1, e_2, \ldots, e_n \}.
\]

 we conclude that the equicontinuity region of $\Gamma$ is \\ 
$\mathrm{Eq}(\Gamma) = L\{ e_1, e_2, \ldots, e_n \}$. 
\end{proof}

\begin{lemma}
\label{th:par2}
Let $\tilde{\gamma} \in \mathrm{PSL}(n+1, \mathbb{H})$ be a parabolic transformation whose lift is given by 
\[
\gamma = 
\begin{bmatrix}
    \mathrm{D}(\lambda_1, \lambda_2, \ldots, \lambda_l) & \\
    & \mathrm{J}(\mu, k)
\end{bmatrix}
\in \mathrm{SL}(n+1, \mathbb{H}).
\]
Then the equicontinuity region of the cyclic group $\Gamma := \langle \tilde{\gamma} \rangle$ is 
\[
\mathrm{Eq}(\Gamma) = \P^n_\H\setminus L\{ e_1, e_2, \ldots, e_n \}.
\]
\end{lemma}

\begin{proof}
Since $\gamma$ is block diagonal, we have
\[
\gamma^m =
\begin{bmatrix}
    \mathrm{D}(\lambda_1, \lambda_2, \ldots, \lambda_l)^m & \\
    & \mathrm{J}(\mu, k)^m
\end{bmatrix}.
\]
Consider the forward sequence
\[
\frac{1}{{m \choose k-1}} \gamma^m.
\]
 then every convergent subsequence of this sequence converging to of the form;
\[
g =
\begin{bmatrix}
    \mathbf{0} & \\
    & a\,\mathrm{E}_{1,k}
\end{bmatrix}  \in \mathrm{QP}(n+1,\mathbb{H})
\]
where $|a| = 1$. Consequently,
\[
\ker(g) = L\{ e_1, e_2, \ldots, e_n \}.
\]

Similarly, consider the backward sequence
\[
\frac{1}{{m + k - 2 \choose k - 1}} \gamma^{-m}.
\]
Also every subsequence of this sequence converges to of the form;
\[
g =
\begin{bmatrix}
    \mathbf{0} & \\
    & a\,\mathrm{E}_{k-1,k}
\end{bmatrix} \in \mathrm{QP}(n+1,\mathbb{H})
\]
where again $|a| = 1$. Hence,
\[
\ker(g) = L\{ e_1, e_2, \ldots, e_n \}.
\]

Combining both cases, we conclude that
$\mathrm{Eq}(\Gamma) = \P^n_\H\setminus L\{ e_1, e_2, \ldots, e_n \}$. 
\end{proof}

\begin{lemma}
\label{th:par3}
Let $\tilde{\gamma} \in \mathrm{PSL}(n+1, \mathbb{H})$ be a parabolic transformation whose lift is given by 
\[
\gamma =
\begin{bmatrix}
    \mathrm{J}(\lambda, k) & \\
    & \mathrm{J}(\lambda, k)
\end{bmatrix}
\in \mathrm{SL}(n+1, \mathbb{H}), \qquad 2k = n + 1.
\]
Then the equicontinuity region of the cyclic group $\Gamma := \langle \tilde{\gamma} \rangle$ is 
\[
\mathrm{Eq}(\Gamma) = \P^n_\H\setminus L\{ e_1, e_2, \ldots, e_{k-1}, e_{k+1}, \ldots, e_{2k-1} \}.
\]
\end{lemma}

\begin{proof}
The $m$-th power of $\gamma$ is given by
\[
\gamma^m =
\begin{bmatrix}
    \mathrm{J}(\lambda, k)^m & \\
    & \mathrm{J}(\lambda, k)^m
\end{bmatrix}.
\]
Consider the sequence
\[
\frac{1}{{m \choose k-1}} \gamma^m.
\]
It admits a subsequence converging to
\[
g =
\begin{bmatrix}
    a\,\mathrm{E}_{1,k} & \\
    & a\mathrm{E}_{1,k}
\end{bmatrix},
\]
where $|a| = 1$. Hence,
\[
\ker(g) = L\{ e_1, e_2, \ldots, e_{k-1}, e_{k+1}, \ldots, e_{2k-1} \}.
\]

Similarly, for the inverse powers $\gamma^{-m}$, the corresponding normalized sequence 
\[
\frac{1}{{m + k - 2 \choose k - 1}} \gamma^{-m}
\]
has a subsequence converging to a matrix of the same form, yielding
\[
\ker(g) = L\{ e_1, e_2, \ldots, e_{k-1}, e_{k+1}, \ldots, e_{2k-1} \}.
\]

Combining both cases, we conclude that
\[
\mathrm{Eq}(\Gamma) = \P^n_\H\setminus L\{ e_1, e_2, \ldots, e_{k-1}, e_{k+1}, \ldots, e_{2k-1} \}.
\]
This completes the proof. 
\end{proof}

\begin{lemma}
\label{th:par4}
Let $\tilde{\gamma} \in \mathrm{PSL}(n+1, \mathbb{H})$ be a parabolic transformation whose lift is given by 
\[
\gamma =
\begin{bmatrix}
    \mathrm{J}(\lambda, k) & \\
    & \mathrm{J}(\lambda, l)
\end{bmatrix}
\in \mathrm{SL}(n+1, \mathbb{H}), \qquad k + l = n + 1, \; k < l.
\]
Then the equicontinuity region of the cyclic group $\Gamma := \langle \tilde{\gamma} \rangle$ is 
\[
\mathrm{Eq}(\Gamma) = \P^n_\H\setminus L\{ e_1, e_2, \ldots, e_{k+l-1} \}.
\]
\end{lemma}

\begin{proof}
The $m$-th power of $\gamma$ is given by
\[
\gamma^m =
\begin{bmatrix}
    \mathrm{J}(\lambda, k)^m & \\
    & \mathrm{J}(\lambda, l)^m
\end{bmatrix}.
\]
Consider the sequence
\[
\frac{1}{{m \choose l - 1}} \gamma^m.
\]
Since \(k < l\), the ratio of binomial coefficients satisfies;
\[
\frac{\binom{m}{\,k-1\,}}{\binom{m}{\,l-1\,}} \longrightarrow 0 
\qquad \text{as } m \to \infty.
\]

and 
\[
\frac{1}{{m \choose l - 1}} \mathrm{J}(\lambda, l)^m \longrightarrow a\,\mathrm{E}_{1,l}
\]
for some $|a| = 1$, it follows that
\[
\frac{1}{{m \choose l - 1}} \gamma^m \longrightarrow
g =
\begin{bmatrix}
    \mathbf{0} & \\
    & a\,\mathrm{E}_{1,l}
\end{bmatrix}  \in \mathrm{QP}(n+1,\mathbb{H})
\]
Hence,
$\ker(g) = L\{ e_1, e_2, \ldots, e_{k+l-1} \}.$

Similarly, for the inverse powers $\gamma^{-m}$, the sequence
$\frac{1}{{m + l - 2 \choose l - 1}} \gamma^{-m}$
converges to the same limit
\[
g =
\begin{bmatrix}
    \mathbf{0} & \\
    & a\,\mathrm{E}_{1,l}
\end{bmatrix} \in \mathrm{QP}(n+1,\mathbb{H})
\]
and therefore,
$\ker(g) = L\{ e_1, e_2, \ldots, e_{k+l-1} \}.$

Combining the two cases, we conclude that
\[
\mathrm{Eq}(\Gamma) =  \P^n_\H\setminus L\{ e_1, e_2, \ldots, e_{k+l-1} \} .
\]
This completes the proof. 
\end{proof}

\begin{theorem}
\label{th:par-main}
Let $\gamma \in \mathrm{PSL}(n+1, \mathbb{H})$ be a parabolic element whose Jordan form consists of 
a diagonal block of size $k_1$, and for $2 \le i \le m$, Jordan blocks of sizes $k_i$ appearing 
$l_i$ times, with $k_i < k_{i+1}$, and satisfying
\[
k_1 + \sum_{i=2}^{m} l_i k_i = n.
\]
Then the equicontinuity region of $\gamma$ is given by
\begin{align*}
\mathrm{Eq}(\gamma) = \P^n_\H\setminus L\Big\{
& e_1, \ldots, 
e_{k_1 + \sum_{i=2}^{m-1} l_i k_i + (k_m - 1)}, \,
e_{k_1 + \sum_{i=2}^{m-1} l_i k_i + (k_m + 1)}, \ldots,
e_{k_1 + \sum_{i=2}^{m-1} l_i k_i + (2k_m - 1)}, \, \\
& e_{k_1 + \sum_{i=2}^{m-1} l_i k_i + (2k_m + 1)}, \ldots, \,
e_{k_1 + \sum_{i=2}^{m-1} l_i k_i + (j k_m - 1)}, \,
e_{k_1 + \sum_{i=2}^{m-1} l_i k_i + (j k_m + 1)}, \ldots, \\
& e_{k_1 + \sum_{i=2}^{m-1} l_i k_i + (l_m k_m - 1)}
\Big\}.
\end{align*}
\end{theorem}

The proof of Theorem \ref{th:par-main} can be done using the lemmas \ref{th:par1}, \ref{th:par2}, \ref{th:par3}, \ref{th:par4}.

\section{Loxodromic Case}\label{sec:loxodromic}
\begin{theorem}
\label{th:loxodromic}
Let $\tilde{\gamma} \in \mathrm{PSL}(n+1,\mathbb{H})$ be a loxodromic element admitting a lift in $\mathrm{SL}(n+1,\mathbb{H})$ of the form
\[
\gamma =
\begin{bsmallmatrix}
    \mathrm{D}(|\lambda_1|, k_1) & & & \\
    & \mathrm{D}(|\lambda_2|, k_2) & & \\
    & & \ddots & \\
    & & & \mathrm{D}(|\lambda_p|, k_p)
\end{bsmallmatrix},
\]
where $k_1 + k_2 + \cdots + k_p = n+1$ and 
$|\lambda_1| < |\lambda_2| < \cdots < |\lambda_p|$.
Then the equicontinuity region of the cyclic group 
$\Gamma := \langle \tilde{\gamma} \rangle$ is
\[
\mathrm{Eq}(\Gamma)
=\P^n_\H\setminus\left( L\{e_1, e_2, \ldots, e_{k_1 + \cdots + k_{p-1}}\}
\, \cup \,
L\{e_{k_1 + 1}, e_{k_1 + 2}, \ldots, e_{n+1}\}\right),
\]
where $\mathrm{D}(|\lambda_i|, k_i)$ denotes a diagonal block of size $k_i$ with diagonal entries of modulus $|\lambda_i|$.
\end{theorem}

\begin{proof}
The $k$-th power of $\gamma$ is given by
\[
\gamma^k =
\begin{bsmallmatrix}
    \mathrm{D}(|\lambda_1|^k, k_1) & & & \\
    & \mathrm{D}(|\lambda_2|^k, k_2) & & \\
    & & \ddots & \\
    & & & \mathrm{D}(|\lambda_p|^k, k_p)
\end{bsmallmatrix}.
\]
If a subsequence of $\frac{1}{|\lambda_p|^k}\gamma^k$, converges then it must be of the form
\[
\gamma' =
\begin{bsmallmatrix}
    \mathbf{0} & & & \\
    & \ddots & & \\
    & & \mathbf{0} & \\
    & & & \mathrm{D}(|\alpha|, k_p)
\end{bsmallmatrix}  \in \mathrm{QP}(n+1,\mathbb{H})
\]
where $|\alpha| = 1$. It follows that 
\[
\ker(\gamma') = L\{e_1, \ldots, e_{k_1 + \cdots + k_{p-1}}\}.
\]

Similarly, every convergent subsequence of $\frac{1}{|\lambda_p|^{-k}}\gamma^{-k}$ converges to
\[
\gamma'' =
\begin{bsmallmatrix}
    \mathrm{D}(|\alpha|, k_1) & & & \\
    & \ddots & & \\
    & & \mathbf{0} & \\
    & & & \mathbf{0}
\end{bsmallmatrix}  \in \mathrm{QP}(n+1,\mathbb{H})
\]
and 
\[
\ker(\gamma'') = L\{e_{k_1 + 1}, \ldots, e_{n+1}\}.
\]
Combining these, we obtain
\[
\mathrm{Eq}(\Gamma)
= \P^n_\H\setminus\left(L\{e_1, e_2, \ldots, e_{k_1 + \cdots + k_{p-1}}\}
\, \cup \,
L\{e_{k_1 + 1}, e_{k_1 + 2}, \ldots, e_{n+1}\}\right).
\]
This completes the proof. 
\end{proof}

\section{Equicontinuity regions of cyclic subgroups generated by loxoparabolic element} \label{sec:loxo}
    \begin{lemma}\label{th:loxo1}
Let $\tilde{\gamma}\in \mathrm{PSL}(n+1,\mathbb{H})$ be loxoparabolic, and let a lift in 
$\mathrm{SL}(n+1,\mathbb{H})$ have the block form
$$
\gamma=
\begin{bmatrix}
\lambda_1\,\mathrm{J}(1,k_1) & 0\\[2mm]
0 & \lambda_2\,\mathrm{J}(1,k_2)
\end{bmatrix},
\qquad 0<|\lambda_1|<|\lambda_2|,\quad k_1+k_2=n+1.
$$
Then the equicontinuity region of the cyclic group $\Gamma=\langle\tilde{\gamma}\rangle$ is
$$
\mathrm{Eq}(\Gamma)
=
\P^n_\H\setminus\left({L}\{e_1,\dots,e_n\}
\ \cup\
{L}\{e_1,\dots,e_{k_1-1},\,e_{k_1+1},\dots,e_{n+1}\}\right).
$$
\end{lemma}

\begin{proof}
By taking power of the Jordan block,
$$
\mathrm{J}(1,k)^m=\left[\binom{m}{\,j-i\,}\right]_{1\le i\le j\le k},
$$
the maximal entry of $\mathrm{J}(1,k_2)^m$ is the $(1,k_2)$-th entry, which is $\binom{m}{k_2-1}$, and of $\mathrm{J}(1,k_1)^m$ is $(1,k_1)$-th entry $\binom{m}{k_1-1}$. Since
$$
\gamma^m=
\begin{bmatrix}
\lambda_1^m \mathrm{J}(1,k_1)^m & 0\\[1mm]
0 & \lambda_2^m \mathrm{J}(1,k_2)^m
\end{bmatrix}.
$$

Normalize $\gamma^m$ by $\lambda_2^m \binom{m}{k_2-1}$. Because $|\lambda_1/\lambda_2|<1$, the upper-left block tends to $0$, and the lower-right block tends entrywise to $\mathrm{E}_{1,k_2}$. Therefore,
$$
\gamma^m \longrightarrow
\gamma^+:=
\begin{bmatrix}
0 & 0\\[1mm]
0 & \mathrm{E}_{1,k_2}
\end{bmatrix} \in \mathrm{QP}(n+1,\mathbb{H}),
$$
so $\gamma^+\in Lim(\langle \tilde{\gamma}\rangle)$, with
$$
\ker(\gamma^+)=L\{e_1,\dots,e_n\}.
$$

Similarly, using
$$
\mathrm{J}(1,k)^{-m}=\left[ (-1)^{j-i}\binom{m+j-i-1}{\,j-i\,}\right]_{i\le j},
$$
the dominant entry of the first block is $(1,k_1)$. Normalizing $\gamma^{-m}$ by $\lambda_1^{-m} \binom{m+k_1-2}{k_1-1}$ gives
$$
\gamma^{-m}\longrightarrow
\gamma^-:=
\begin{bmatrix}
\mathrm{E}_{1,k_1} & 0\\[1mm]
0 & 0
\end{bmatrix} \in \mathrm{QP}(n+1,\mathbb{H}),
$$
so $\gamma^-\in Lim(\langle \tilde{\gamma}\rangle)$, with
$$
\ker(\gamma^-)={L}\{e_1,\dots,e_{k_1-1}, e_{k_1+1},\dots,e_{n+1}\}.
$$

Therefore the equicontinuity region is the complement of the union of the kernel hyperplanes of the forward and backward pseudo-projective limits:
$$
\mathrm{Eq}(\Gamma)=\P^n_\H\setminus\left(\ker(\gamma^+)\cup \ker(\gamma^-)\right)
=\P^n_\H\setminus\left({L}\{e_1,\dots,e_n\}\ \cup\ {L}\{e_1,\dots,e_{k_1-1}, e_{k_1+1},\dots,e_{n+1}\}\right),
$$
as claimed.
\end{proof}

\begin{theorem}
\label{th:loxoparabolic}
Let $\gamma \in \mathrm{PSL}(n+1,\mathbb{H})$ be loxoparabolic transformation such that lift of $\gamma$ in $\mathrm{SL}(n+1,\H)$ still denoted by $\gamma$ is of the form;
\[
\gamma =
\begin{bsmallmatrix}
\lambda_1 A_1 & & \\
& \ddots & \\
& & \lambda_k A_k
\end{bsmallmatrix},
\qquad |\lambda_1|<|\lambda_2|<\cdots<|\lambda_k|,
\]
where each $A_i$ is a parabolic block of size $m_i\times m_i$ for $i=1,2,\ldots, k$.
Let $\Gamma=\langle\gamma\rangle$ be the cyclic subgroup generated by $\gamma$.
Then the equicontinuity region of $\Gamma$ is
\[
\mathrm{Eq}(\Gamma)
=
\P^n_\H\setminus\left({L}\Big\{
\mathrm{Eq}(A_1),\;
e_{m_1+1},\;\dots,\;e_{n+1},
\Big\} \cup {L}\Big\{
e_1, \dots,
e_{m_{k-1}},\mathrm{Eq}(A_k)
\Big\}\right)
\]
where
$\mathrm{Eq}(A_i)$ denotes the equicontinuity region of the restriction of the parabolic block $A_i$.
\end{theorem}

The proof of Theorem \ref{th:loxoparabolic} follows from Lemma \ref{th:loxo1}, \ref{th:par1}, \ref{th:par2}, \ref{th:par3}, \ref{th:par4} and Theorem \ref{th:loxodromic}.

    \section*{Declaration of competent interest }
    The authors declare no competing interests.
    \section*{Data availability}
    No data was used for the research described in the article.
    \section*{Acknowledgments}
		Dutta acknowledges the Mizoram University for the Research Promotion Grant F.No.A.1-1/MZU(Acad)/14/25-26. Gongopadhyay acknowledges ANRF research Grant CRG/2022/003680 and ANRF/ARGM/2025/000122/MTR. Mondal acknowledges the  CSIR grant no. 09/0947(12987)/2021-EMR-I during the course of this work.

    \end{document}